\documentclass[11pt,a4paper]{amsart}  
\usepackage{comment}
\usepackage{tikz-cd}
\usepackage{amsmath,amsthm,amssymb,enumerate,amsfonts,graphicx}

\usepackage{subcaption}
\usepackage{algorithm,algorithmic}

\usepackage{color}
\usepackage{epsfig}
\usepackage[numbers,sort&compress]{natbib}
\usepackage[T1]{fontenc}
\usepackage{algorithm,algorithmic}
\usepackage{stmaryrd}
\usepackage{paralist}
\usepackage{geometry}
\usepackage{tikz}
\usepackage{todonotes}
\tikzset{black node/.style={draw, circle, fill = black, minimum size = 5pt, inner sep = 0pt}}
\tikzset{normal/.style = {draw=none, fill = none, minimum size =0, rectangle}}

\makeatletter
\newtheorem*{rep@theorem}{\rep@title}
\newcommand{\newreptheorem}[2]{%
\newenvironment{rep#1}[1]{%
 \def\rep@title{#2 \ref{##1}}%
 \begin{rep@theorem}}%
 {\end{rep@theorem}}}
\makeatother

\newreptheorem{theorem}{Theorem}

\usepackage{aliascnt}

\newtheorem{theorem}{Theorem}%[section]

\newaliascnt{lemma}{theorem}
\newtheorem{lemma}[lemma]{Lemma}
\aliascntresetthe{lemma}

\newaliascnt{observation}{theorem}

\aliascntresetthe{observation}

\newaliascnt{corollary}{theorem}
\newtheorem{corollary}[corollary]{Corollary}
\aliascntresetthe{corollary}

\newaliascnt{conjecture}{theorem}

\aliascntresetthe{conjecture}

\newaliascnt{claim}{theorem}

\aliascntresetthe{claim}

\theoremstyle{definition}

\newcommand{\R}{\mathbb{R}}
\newcommand{\N}{\mathbb{N}}
\newcommand{\Z}{\mathbb{Z}}
\newcommand{\Q}{\mathbb{Q}}

\DeclareMathOperator{\OPT}{\mathrm{OPT}}

\DeclareMathOperator{\SA}{\mathsf{SA}}

\newcommand{\fvst}{\textsc{FVST}}

\begin{document}

\title[A simple $7/3 $-approximation algorithm for FVST]{A simple $7/3 $-approximation algorithm\\ for feedback vertex set in tournaments}

\author[M.~Aprile]{Manuel Aprile}
\author[M.~Drescher]{Matthew Drescher}
\author[S.~Fiorini]{Samuel Fiorini}
\author[T.~Huynh]{Tony Huynh}
\address[M.~Aprile, M.~Drescher, S.~Fiorini]{\newline D\'epartement de Math\'ematique
\newline Universit\'e libre de Bruxelles
\newline Brussels, Belgium}
\email{manuelf.aprile@gmail.com, knavely@gmail.com, sfiorini@ulb.ac.be}
\address[T.~Huynh]{\newline School of Mathematics
\newline Monash University
\newline Melbourne, Australia}
\email{tony.bourbaki@gmail.com}

\thanks{This project was supported by ERC Consolidator Grant 615640-ForEFront. Samuel Fiorini and Manuel Aprile are also supported by FNRS grant T008720F-35293308-BD-OCP. Tony Huynh is also supported by the Australian Research Council.}

\date{\today}
\sloppy

\begin{abstract}
We show that performing just one round of the Sherali-Adams hierarchy gives an easy $7/3$-approximation algorithm for the Feedback Vertex Set (\fvst) problem in tournaments. This matches the best deterministic approximation algorithm for \fvst{} due to Mnich, Williams, and V{\'{e}}gh~\cite{MWV16}, and is a significant simplification and runtime improvement of their approach.  
\end{abstract}

\maketitle

\section{Introduction}
A \emph{feedback vertex set} (FVS) of a tournament $T$ is a set $X$ of vertices such that $T-X$ is acyclic. Given a tournament $T$ and (vertex) weights $w : V(T) \to \mathbb{Q}_{\geq 0}$, the \emph{Feedback Vertex Set} (\fvst) problem asks to find a feedback vertex set $X$ such that $w(X):=\sum_{x \in X} w(x)$ is minimum.  This problem has numerous applications, for example in determining election winners in social choice theory~\cite{banks85}.

We let $\OPT(T, w)$ be the minimum weight of a feedback vertex set of the weighted tournament $(T,w)$. An \emph{$\alpha$-approximation algorithm} for \fvst{} is a polynomial-time algorithm computing a feedback vertex set $X$ with $w(X) \leq \alpha \cdot \OPT(T,w)$.

Note that a tournament is acyclic if and only if it does not contain a directed triangle.  Therefore, the following is an easy $3$-approximation algorithm for \fvst{} in the unweighted case (the general case follows for instance from the \emph{local ratio technique}~\cite{bbfr2004}). If $T$ is acyclic, then $\varnothing$ is an FVS, and we are done.  Otherwise, we find a directed triangle $abc$ in $T$ and put all its vertices into the FVS. We then replace $T$ by $T-\{a,b,c\}$ and recurse. 

\subsection*{State of the Art}

The first non-trivial approximation algorithm for $\fvst{}$ was a $5/2$-approximation algorithm by Cai, Deng, and Zang~\cite{CDZ01}. Cai \emph{et al.}'s approach is polyhedral. It is based on the fact that the basic LP relaxation of \fvst{} is integral whenever the input tournament avoids certain subtournaments, see the next paragraphs for details.

Let $T$ be a tournament and $\triangle(T)$ denote the collection of all $\{a,b,c\} \subseteq V(T)$ that induce a directed triangle in $T$. The \emph{basic relaxation} for $T$ is the polytope 
$$
P(T) := \{ x \in [0,1]^{V(T)} \mid \forall \{a,b,c\} \in \triangle(T) : x_a + x_b + x_c \geq 1 \}.  
$$

Let $\mathcal T_5$ be the set of tournaments on $5$ vertices where the minimum FVS has size $2$. Up to isomorphism, $|\mathcal T_5|=3$ (see~\cite{CDZ01}). We say that $T$ is \emph{$\mathcal{T}_5$-free} if no subtournament of $T$ is isomorphic to a member of $\mathcal{T}_5$. More generally, let $\mathcal T$ be a collection of tournaments.  A \emph{$\mathcal{T}$-subtournament} of $T$ is a subtournament of $T$ that is isomorphic to some tournament of $\mathcal{T}$. We say that $T$ is \emph{$\mathcal T$-free} if $T$ does not contain a $\mathcal{T}$-subtournament.

Cai \emph{et al.} prove that $P(T)$ is integral as soon as $T$ is $\mathcal{T}_5$-free. In this case solving a polynomial-size LP gives a minimum weight FVS.  We let 
\textsc{CDZ$(T,w)$}, be the polynomial-time algorithm from~\cite{CDZ01}, that given a $\mathcal{T}_5$-free tournament $T$ and $w: V(T) \to \mathbb{Q}_{\geq 0}$, finds a minimum weight feedback vertex set of $T$.

A $5/2$-approximation algorithm follows directly from this. Using the local ratio technique, while $T$ contains a $\mathcal T_5$-subtournament $S$, one can reduce to a smaller instance with one vertex of $S$ removed. If one is aiming for a $5/2$-approximation algorithm, one can reduce to a $\mathcal{T}_5$-free tournament $T$, for which one can even solve the problem exactly by applying \textsc{CDZ$(T,w)$}.

The $5/2$-approximation algorithm of \cite{CDZ01} was improved to a $7/3$-approximation algorithm by Mnich, Williams, and V{\'{e}}gh~\cite{MWV16}. Loosely speaking, Mnich \emph{et al.}'s algorithm replaces $\mathcal{T}_5$ by $\mathcal{T}_7$, defined as the set of tournaments on $7$ vertices where the minimum FVS has size $3$. It is known that, up to isomorphism, $|\mathcal T_7|=121$ (see~\cite{MWV16}). 

Similarly, if one is aiming for a $7/3$-approximation algorithm, one can reduce to $\mathcal{T}_7$-free tournaments. In fact, instead of using the local ratio technique, \cite{MWV16} use iterative rounding, see the next paragraph. However, the basic relaxation is not necessarily integral for $\mathcal T_7$-free tournaments, so obtaining a $7/3$-approximation algorithm requires more work.

The algorithm in~\cite{MWV16} consists of two phases. Let the \emph{$\mathcal T_7$-relaxation} be the LP  obtained from the basic relaxation by adding the constraint $\sum_{v \in V(S)} x_v \geq 3$ for each $\mathcal{T}_7$-subtournament $S$ of $T$.
The first phase is an iterative rounding procedure on the $\mathcal T_7$-relaxation. This reduces the problem to a residual tournament which is $\mathcal T_7$-free.  The second phase is a $7/3$-approximation algorithm for FVST on the residual tournament, via an intricate layering procedure.

Recently, Lokshtanov, Misra, Mukherjee, Panolan, Philip, and Saurab~\cite{LMMPPS20} gave a \emph{randomized} $2$-approximation algorithm for \fvst{}. Their algorithm does not rely on~\cite{CDZ01}, but rather on the idea of guessing vertices which are not part of some optimal FVS and that of controlling the in-degree sequence of the tournament. The derandomized version of their algorithm runs in quasi-polynomial-time. A deterministic $2$-approximation algorithm would be best possible, since for every $\epsilon>0$, $\fvst{}$ does not have a $(2-\epsilon)$-approximation algorithm, unless the Unique Games Conjecture is false or P=NP~\cite{speckenmeyer89, KR08}.

\subsection*{Our Contribution}

We simplify Mnich \emph{et al.}'s $7/3$-approximation algorithm for \fvst{}~\cite{MWV16}. Our new algorithm is based on performing just one round of the \emph{Sherali-Adams hierarchy}~\cite{SA1990} on the basic relaxation, and is a significant simplification of~\cite{MWV16}. The following is our main theorem. Below, $\SA_r(T,w)$ denotes both the lower bound on $\OPT(T,w)$ provided by $r$ rounds of the Sherali-Adams hierarchy, and the corresponding linear program (LP).

\begin{theorem} \label{thm:main}
Algorithm \ref{alg:main} is a $7/3$-approximation algorithm for \fvst. More precisely, the algorithm outputs in polynomial time a feedback vertex set $X := F \cup F'$ such that $w(X) \leq \frac{7}{3} \SA_1(T,w) \leq \frac{7}{3} \OPT(T,w)$. 
\end{theorem}

\begin{algorithm}\caption{\sc{FVST}}\label{alg:main} 
\begin{algorithmic}[1]
\REQUIRE Tournament $T$ and weight function $w: V(T) \rightarrow \mathbb{Q}_{> 0}$ 
\ENSURE A feedback vertex set of $T$ of weight at most $\frac{7}{3}\OPT(T,w)$%$\OPT_{\mathbb{Z}}(T)$ 
\STATE{$x \gets $ optimal solution to $\SA_{1}(T,w)$}
\STATE{$F \gets \{v \in V(T) : x_v \geq 3/7 \}$}
\IF{$F$ is a FVS for $T$}
    \STATE{return $F$} \label{allround}
\ELSE
    \STATE{$Z \gets \varnothing$}
    \REPEAT \label{line:repeat}
        \STATE{add to $Z$ all vertices of $T-F-Z$ that are contained in no triangle}
        \STATE{$x \gets $optimal solution to $\SA_0(T-F-Z,w)$}
        \STATE{$F \gets F \cup \{v \in V(T-F-Z) : x_v \geq 1/2\}$} \label{line:round_up_half}
    \UNTIL{$T-F-Z$ is empty or $x_v < 1/2$ for all $v \in V(T-F-Z)$} \label{line:until}
    \STATE{$F' \gets$ \sc{Layers}$(T-F-Z,w,\varnothing,V(T-F-Z))$ }
    \STATE{return $F \cup F'$}\label{terminate}
\ENDIF
\end{algorithmic}
\end{algorithm}

Theorem~\ref{thm:main} proves that the integrality gap of the relaxation obtained from the basic one after one round of Sherali-Adams is always at most $7/3$. We observe that for random unweighted tournaments $(T,\mathbf{1}_T)$, letting $x_v := 3/7$ for all vertices always gives a feasible solution while the optimum value is with high probability very close to $|V(T)|$, see Corollary~\ref{cor:integrality_gap}. Thus the worst case integrality gap of $\SA_1$ is precisely $7/3$.

Precise definitions will be given later.  For now, we give a sketch of Algorithm~\ref{alg:main}, and explain how it compares with~\cite{MWV16}. 

\subsection*{Comparison to Previous Work}

Our approach simplifies both phases of Mnich \emph{et al.}'s algorithm~\cite{MWV16}. In our first phase (the rounding phase), instead of considering the $\mathcal T_7$-relaxation, we consider $\SA_{1}(T,w)$. Since the rounding phase is the bottleneck of both algorithms, we obtain a significant speedup in run-time by using a smaller LP.  Note that $\SA_{1}(T,w)$ only has $O(n^4)$ constraints, while the $\mathcal T_7$-relaxation can have $\Omega(n^7)$ constraints. Let $x$ be an optimal solution to $\SA_{1}(T,w)$.  If $x$ has a coordinate $x_v$ such that $x_v \geq 3/7$, then we may round up $x_v$ to $1$. We continue the rounding using $\SA_{0}(T,w)$ (the basic relaxation) instead, in order to make sure that when we start the second phase (the layering phase), in the residual tournament, the optimum value is at least one third of the total weight. The whole rounding is done exactly as in~\cite{MWV16}, except that we replace the $\mathcal{T}_7$-relaxation with $\SA_1(T,w)$.

Then, we proceed to the second phase. The idea follows~\cite{MWV16}, but with a few important simplifications. We start from a minimum in-degree vertex $z$ and build a breadth-first search (BFS) in-arborescence that partitions $V(T)$ in layers such that every triangle of $T$ lies within three consecutive layers. Hence, a feedback vertex set for $T$ can be obtained by including every other layer, and, for every layer $i$ that is not picked, a set $F_i$ that is a feedback vertex set for that layer (we call the set $F_i$ a \emph{local solution}). 

The main difference with the layering algorithm of \cite{MWV16} is how local solutions are selected. The layers obtained by the algorithm in \cite{MWV16} are $\mathcal{T}_5$-free. This allows them to use \textsc{CDZ$(T,w)$} as a subroutine to optimally select local solutions. 
Our algorithm implements a simpler procedure to partition $V(T)$ in layers.
For the first layer produced by the BFS procedure, consisting of all vertices that point to $z$, we also use \textsc{CDZ$(T,w)$}. However, for the subsequent layers, a different property is established. 

Such layers can be partitioned into two subtournaments, $U_i$ and $S_i$, that are both acyclic. Hence, we can choose the cheaper of the two subtournaments as our local solution $F_i$. Whenever the BFS procedure is stuck, that is, when none of the remaining vertices can reach the root node $z$, the algorithm chooses another root node and starts again (we refer to this as a \emph{fresh start}). Our method gives an improved $9/4$-approximation algorithm for FVST on our residual tournament, compared to the $7/3$ factor obtained in~\cite{MWV16}.

\subsection*{Paper Outline}
In Section~\ref{sec:Sherali-Adams}, we define the Sherali-Adams hierarchy. We introduce a local structure called a \emph{diagonal} in Section~\ref{sec:diagonals}, which will be be helpful in our rounding procedure.  We also classify every tournament as either \emph{light} or \emph{heavy}, and derive some structural properties of light tournaments.  These results will be used later, since the input of our layering algorithm is a light tournament. In Section~\ref{sec:layering}, we describe our layering procedure.  Finally, in Section~\ref{sec:algorithm}, we state Algorithm~\ref{alg:main} in full and prove its correctness. A conclusion is given in Section~\ref{sec:conclusion}.

\section{The Sherali-Adams Hierarchy} \label{sec:Sherali-Adams}
Let $P =  \{ x \in \R^n \mid A x \geq b \}$ be a polytope contained in $[0,1]^n$ and $P_I := \mathrm{conv}(P \cap \Z^n)$.  
Numerous optimization problems can be formulated as minimizing a linear function over $P_I$, where $P$ has only a polynomial number of constraints. For example, let $T$ be a tournament and $w:V(T) \to \mathbb{Q}_{\geq 0}$.  Then $\OPT(T,w)$ is simply the minimum of $w^{\intercal}x$ over $P_I$, where 
$P = P(T)$ is the basic relaxation defined above. 

The Sherali-Adams hierarchy~\cite{SA1990} is a simple but powerful method to obtain improved \emph{approximations} for $P_I$. Since it does not require any knowledge of the structure of $P_I$, it is widely applicable.  The procedure comes with a parameter $r$, which specifies the accuracy of the approximation. That is, for each $r \in \mathbb N$, we define a polytope $\SA_r(P)$. These polytopes satisfy $P = \SA_0(P) \supseteq \SA_1(P) \supseteq \dots \supseteq \SA_r(P) \supseteq \dots \supseteq P_I$.  

An important property of the procedure is that if $P$ is described by a polynomial number of constraints and $r$ is a constant, then $\SA_r(P)$ is also described by a polynomial number of constraints (in a higher dimensional space).  Therefore, for NP-hard optimization problems (such as \fvst{}), one should not expect that $\SA_r(P)=P_I$ for some constant $r$.  However, as we will see, good approximations of $P_I$ can be extremely useful if we want to \emph{approximately} optimize over $P_I$.  Indeed, despite some recent results~\cite{YZ14, levey20191+, garg2018quasi, OS19, ADFH20}, we feel that the Sherali-Adams hierarchy is underutilized in the design of approximation algorithms, and hope that our work will inspire further applications.    

Here is a formal description of the procedure.  Let $P =  \{ x \in \R^n \mid A x \geq b \} \subseteq [0,1]^n$ and $r \in \N$.  Let $N_r$ be the \emph{nonlinear} system obtained from $P$ by multiplying each constraint by $\prod_{i \in I} x_i \prod_{j \in J}(1-x_j)$ for all disjoint subsets $I, J$ of $[n]$ such that $1 \leq |I|+|J| \leq r$. Note that if $x_i \in \{0,1\}$, then $x_i^2=x_i$.  Therefore, we can obtain a \emph{linear} system $L_r$ from $N_r$ by setting $x_i^2:=x_i$ for all $i \in [n]$ and then $x_I:=\prod_{i \in I} x_i$ for all $I \subseteq [n]$ with $|I| \geq 2$.  We then let $\SA_r(P)$ be the projection of $L_r$ onto the variables $x_i$, $i \in [n]$.  

We let $\SA_r(T):=\SA_r(P(T))$, where $P(T)$ is the basic relaxation. 

For the remainder of the paper, we only need the inequalities defining $\SA_1(T)$, which we now describe.  Recall that $\triangle(T)$ is the collection of all $\{a,b,c\} \subseteq V(T)$ that induce a directed triangle in $T$. We call the elements of $\triangle(T)$ \emph{triangles}. For all $\{a,b,c\} \in \triangle(T)$ and $d \in V(T-a-b-c)$, we have the inequalities
\begin{align}
\label{ineq:SAtype1} x_a + x_b + x_c &\geq 1 + x_{ab} + x_{bc}\,,\\
\label{ineq:SAtype2} x_{ad} + x_{bd} + x_{cd} &\geq x_d \quad \text{and}\\
\label{ineq:SAtype3} x_a + x_b + x_c + x_{d} &\geq 1 + x_{ad} + x_{bd} + x_{cd}\,.
\end{align}
In addition, there are the inequalities
\begin{equation}
\label{ineq:SAbounds}
1 \geq x_{a} \geq x_{ab} \geq 0
\end{equation}
for all distinct $a, b \in V(T)$.  Let $E(T)$ be the set of all unordered pairs of vertices of $T$. The polytope $\SA_1(T)$ is the set of all $(x_a)_{a \in V(T)} \in \R^{V(T)}$ such that there exists $(x_{ab})_{ab \in E(T)} \in \R^{E(T)}$ so that inequalities \eqref{ineq:SAtype1}--\eqref{ineq:SAbounds} are satisfied.

\section{Diagonals and Light Tournaments} \label{sec:diagonals}

Let $T$ be a tournament. An (unordered) pair of vertices $ab$ is a \emph{diagonal} if there are vertices $u,v$ such that $\{u,v,a\} \in \triangle(T)$ and $\{u,v,b\} \in \triangle(T)$. We often will denote a triangle $\{a,b,c\}$ as $abc$.  We say that a triangle \emph{contains a diagonal} if at least one of its pairs of vertices is a diagonal, and a triangle is \emph{heavy} if it contains at least two diagonals.  A tournament $T$ is \emph{heavy} if at least one of its triangles is heavy. If a tournament is not heavy, we say that it is \emph{light}.

\begin{lemma}\label{lem:round}
Let $T$ be a tournament and $x \in \SA_1(T)$. If $x_v < 3/7$ for all $v \in V(T)$, % there is no $v \in V(T)$ with $x_v \geq 3/7$
then $T$ is light. 
\end{lemma}
\begin{proof}
First, let $ab$ be a diagonal of $T$. We claim that $x_{ab} \geq 1/7$.
Indeed, since $ab$ is a diagonal there must be $u,v \in V(T)$ with $uva, uvb \in \triangle(T)$. From \eqref{ineq:SAtype1}, $x_{a}+x_u + x_v \geq 1 + x_{au} + x_{av}$ and from \eqref{ineq:SAtype2}, $x_{ab} + x_{au} + x_{av} \geq x_{a}$. 
Adding these two inequalities, we obtain $x_{u} + x_{v} + x_{ab} \geq 1$, implying our claim. 

Now, suppose by contradiction that $T$ is a heavy tournament. Hence there exists $abc \in \triangle(T)$ such that $ab$ and $bc$ are diagonals. By (\ref{ineq:SAtype1}), we have $x_a + x_b + x_c \geq 1 + x_{ab} + x_{bc}$. By the above claim, $x_{ab} \geq 1/7$ and $x_{bc} \geq 1/7$, making the right hand side at least $9/7$. So $\max(x_a,x_b,x_c) \geq 3/7$, a contradiction. 
\end{proof}

Next we prove some results connecting light tournaments to the work of \cite{MWV16}, which relies on tournaments being $\mathcal{T}_7$-free.  
Of the three tournaments in $\mathcal {T}_5$, it turns out one of them is heavy (see Figure~\ref{fig:heavyT5}), while the other two are light and can be obtained from each other by reversing the orientation of one arc (see Figure~\ref{fig:lightT5}).  Moreover, although we do not use this fact, we have a computer-assisted proof which shows that 120 out of 121 of the tournaments in $\mathcal{T}_7$ are heavy, and only one is light.
Thus, even though a light tournament is not necessarily $\mathcal T_7$-free, the property of being light forbids almost all of the tournaments in $\mathcal{T}_7$ as subtournaments.

\begin{figure}[ht]
\centering
\begin{subfigure}[t]{.42\textwidth}\centering
\begin{tikzpicture}[scale=.75, inner sep=1.5pt]
\tikzstyle{vtx}=[circle,draw,thick];
\node[vtx] (a) at (-2,0) {\footnotesize $a$};
\node[vtx] (b) at (2,0) {\footnotesize $b$};
\node[vtx] (c) at (-2,-2) {\footnotesize $c$};
\node[vtx] (d) at (0,-2) {\footnotesize $d$};
\node[vtx] (e) at (2,-2) {\footnotesize $e$};

\draw[thick,->,>=latex] (e) -- (a);
\draw[thick,->,>=latex] (e) to[bend left = 50] (c);
\draw[thick,->,>=latex] (d) -- (e);
\draw[thick,->,>=latex] (b) -- (e);
\draw[thick,->,>=latex] (c) -- (a);
\draw[thick,->,>=latex] (c) -- (b);
\draw[thick,->,>=latex] (d) -- (c);
\draw[thick,->,>=latex] (a) -- (d);
\draw[thick,->,>=latex] (b) -- (d);
\end{tikzpicture}
\caption{Each orientation of $ab$ gives a light tournament in $\mathcal{T}_5$.}
\label{fig:lightT5}
\end{subfigure}
\begin{subfigure}[t]{.42\textwidth}\centering
\begin{tikzpicture}[scale=.75, inner sep=1.5pt]
\tikzstyle{vtx}=[circle,draw,thick];
\node[vtx] (aa) at (4,0) {\footnotesize $a$};
\node[vtx] (bb) at (8,0) {\footnotesize $b$};
\node[vtx] (cc) at (4,-2) {\footnotesize $c$};
\node[vtx] (dd) at (6,-2) {\footnotesize $d$};
\node[vtx] (ee) at (8,-2) {\footnotesize $e$};

\draw[thick,->,>=latex] (aa) -- (bb);
\draw[thick,->,>=latex] (ee) to[bend left = 50] (cc);
\draw[thick,->,>=latex] (dd) -- (ee);
\draw[thick,->,>=latex] (bb) -- (ee);
\draw[thick,->,>=latex] (cc) -- (aa);
\draw[thick,->,>=latex] (bb) -- (cc);
\draw[thick,->,>=latex] (cc) -- (dd);
\draw[thick,->,>=latex] (dd) -- (aa);
\draw[thick,->,>=latex] (bb) -- (dd);
\draw[thick,->,>=latex] (ee) -- (aa);
%\draw[thick] (2,2) -- (2.75,3);
\end{tikzpicture}
\caption{The unique heavy tournament in $\mathcal{T}_5$. Note that triangle $dec$ is heavy.}
\label{fig:heavyT5}
\end{subfigure}
\caption{The three tournaments in $\mathcal{T}_5$.}
\end{figure}
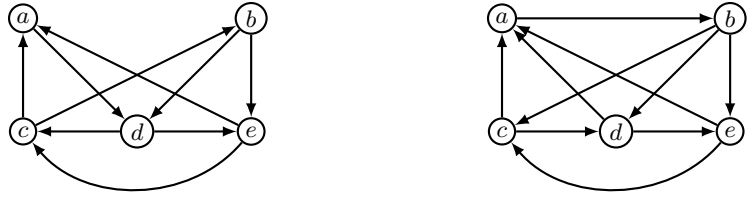

We now establish further properties of light tournaments.  Let $\mathcal S_5 \subseteq \mathcal T_5$ and $\mathcal S_7 \subseteq \mathcal T_7$ be the collection of tournaments defined in Figures~\ref{fig:S5} and~\ref{fig:S7}, respectively. If $T$ is a tournament, we let $A(T)$ be the set of arcs of $T$.

\begin{lemma}
\label{lem:S_5}
Every $S \in \mathcal{S}_5$ is either heavy or has $(u_i,u_{3-i}), (v_i,v_{3-i}) \in A(S)$ for some $i \in [2]$ (where $S$ is labelled as in Figure~\ref{fig:S5}).  
\end{lemma}

\begin{proof}
Suppose $(u_1,u_2), (v_2,v_1) \in A(S)$. Observe that $zv_2$ is a diagonal since $v_1u_1z$ and $v_1u_1v_2$ are triangles, and $v_2u_2$ is a diagonal since $v_1u_1v_2$ and $v_1u_1u_2$ are triangles. Because $zv_2$ and $v_2u_2$ are both diagonals, we conclude that the triangle $v_2u_2z$ is heavy. The result follows by symmetry.
\end{proof}

\begin{figure}[H]
    \centering
    \begin{tikzpicture}[scale=.75,inner sep=1.5pt]
    \tikzstyle{vtx}=[circle,draw,thick];
    \node[vtx] (z) at (0,1.5) {\footnotesize $z$};
    \node[vtx] (u1) at (-1,0) {\footnotesize $u_1$};
    \node[vtx] (u2) at (1,0) {\footnotesize $u_2$};
    \node[vtx] (v1) at (-1,-2) {\footnotesize $v_1$};
    \node[vtx] (v2) at (1,-2) {\footnotesize $v_2$};
    \draw[thick,->,>=latex] (u1) -- (z);
    \draw[thick,->,>=latex] (u2) -- (z);
    \draw[thick,->,>=latex] (v1) -- (u1);
    \draw[thick,->,>=latex] (v2) -- (u2);
    %\draw[thick,->,>=latex] (z) to[bend left=100] (v3);
    \draw[thick,->,>=latex] (z) to[bend right=75] (v1);
    \draw[thick,->,>=latex] (z) to[bend left=75] (v2);
    \draw[thick,->,>=latex] (u1) -- (v2);
    \draw[thick,->,>=latex] (u2) -- (v1);
    %\draw[thick] (2,2) -- (2.75,3);
    \end{tikzpicture}
    \caption{$\mathcal{S}_5$ is the following subset of $\mathcal{T}_5$, where the missing arcs can be oriented arbitrarily.}
    \label{fig:S5}
\end{figure}
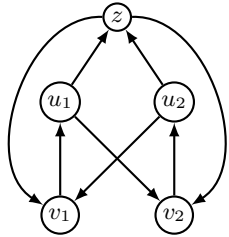

\begin{lemma} \label{lem:S_7}
Every $S \in \mathcal{S}_7$ is heavy.
\end{lemma}

\begin{proof}
Suppose some $S \in \mathcal{S}_7$ is light, where $S$ is labelled as in Figure~\ref{fig:S7}. By symmetry, we may assume that $(u_1,u_2), (u_2,u_3) \in A(S)$. By Lemma~\ref{lem:S_5}, $(v_1,v_2), (v_2,v_3) \in A(S)$. Therefore, $u_2z$ is a diagonal since $v_1u_1z$ and $v_1u_1u_2$ are triangles, and $zv_2$ is a diagonal since $v_3u_3z$ and $v_3u_3v_2$ are triangles. We conclude that $v_2u_2z$ is a heavy triangle, which contradicts that $S$ is light.   
\end{proof}

  \begin{figure}[H]
    \centering
    \begin{tikzpicture}[scale=.75,inner sep=1.5pt]
    \tikzstyle{vtx}=[circle,draw,thick];
    \node[vtx] (z) at (0,2) {\footnotesize $z$};
    \node[vtx] (u2) at (0,0) {\footnotesize $u_2$};
    \node[vtx] (u3) at (3,0) {\footnotesize $u_3$};
    \node[vtx] (u1) at (-3,0) {\footnotesize $u_1$};
    \node[vtx] (v2) at (0,-3) {\footnotesize $v_2$};
    \node[vtx] (v3) at (3,-3) {\footnotesize $v_3$};
    \node[vtx] (v1) at (-3,-3) {\footnotesize $v_1$};
    \draw[thick,->,>=latex] (u1) -- (z);
    \draw[thick,->,>=latex] (u2) -- (z);
    \draw[thick,->,>=latex] (u3) -- (z);
    \draw[thick,->,>=latex] (v1) -- (u1);
    \draw[thick,->,>=latex] (v2) -- (u2);
    \draw[thick,->,>=latex] (v3) -- (u3);
    \draw[thick,->,>=latex] (z) to[bend left=100] (v3);
    \draw[thick,->,>=latex] (z) to[bend right=100] (v1);
    \draw[thick,->,>=latex] (z) to[bend right] (v2);
    \draw[thick,->,>=latex] (u1) -- (v3);
    \draw[thick,->,>=latex] (u1) -- (v2);
    \draw[thick,->,>=latex] (u2) -- (v1);
    \draw[thick,->,>=latex] (u2) -- (v3);
    \draw[thick,->,>=latex] (u3) -- (v1);
    \draw[thick,->,>=latex] (u3) -- (v2);
    %\draw[thick] (2,2) -- (2.75,3);
    \end{tikzpicture}
    \caption{$\mathcal{S}_7$ is the following subset of $\mathcal{T}_7$, where the missing arcs can be oriented arbitrarily.}
    \label{fig:S7}
    \end{figure}
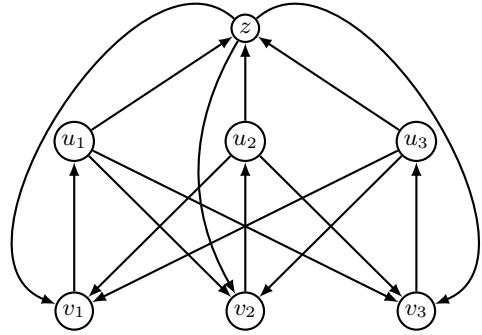

\section{The Layering Procedure}\label{sec:layering}
This section proves the correctness of our layering algorithm, see Algorithm~\ref{layersalgo} below. Lemmas \ref{lem:2-in-dom_implies_Triangle-free} to \ref{lem:Ui-no-T5} ensure that the algorithm actually produces a feedback vertex set.
Lemmas \ref{lem:halfweight} to \ref{lem:correct} prove that Algorithm~\ref{layersalgo} is a $9/4$-approximation algorithm.  

 Let $T$ be a light tournament with weight function $w:  V(T) \rightarrow \mathbb{Q}_{\geq 0}$. For $S \subseteq V(T)$, the \emph{in-neighborhood} of $S$ is $N(S) := \{ v \notin S \mid \text{ $(v,u) \in A(T)$ for some $u \in S$}\}$ and $N(u) := N(\{u\})$. For every $z \in V(T)$, define $V_1(z) =\{z\}$, and for $i \geq 2$ let $V_{i+1}(z):=N(\bigcup_{j \in [i]} V_j(z))$.
In other words $V_i(z)$ is the set of vertices whose shortest directed path to $z$ has length exactly $i-1$.

Given two sets $S,Z \subseteq V(T)$, we say that $Z$ \emph{in-dominates} $S$ 
if for every $s \in S$ there is a $z \in Z$ with $(s,z) \in A(T)$. We say that $Z$ \emph{2-in-dominates} $S$ if $Z$ has a subset $Z' \subseteq Z$ with $|Z'| \leq 2$ such that $Z'$ in-dominates $S$. 

We start with a lemma that is key to both the correctness and the performance guarantee of Algorithm \ref{layersalgo}.
\begin{lemma} \label{lem:2-in-dom_implies_Triangle-free}
Let $T$ be a light tournament, $z$ be any vertex of $T$, and $i \geq 3$. If $V_i(z)$ is $2$-in-dominated by $\{z_{i-1},z_{i-1}'\} \subseteq V_{i-1}(z)$ (possibly $z_{i-1}=z_{i-1}'$), then $U := N(z_{i-1}) \cap V_{i}(z)$ and $S :=V_i(z) - U$ are triangle-free.  
\end{lemma}
\begin{proof}
%Let $U := N(z_{i-1}) \cap V_i(z)$ and $S := V_i(z) - U$.
Suppose by contradiction that $u_1u_2u_3$ is a triangle in $U$. Since  $z_{i-1} \in V_{i-1}(z)$ and $i \geq 3$, we have $(z_{i-1},r) \in A(T)$ for some $r \in V_{i-2}(z)$. Since $U \subseteq V_{i}(z)$, arcs $(r,u_1),(r,u_2),(r,u_3) \in A(T)$. Thus, $ru_iz_{i-1}$ is a triangle  for all $i \in [3]$. It follows that the triangle $u_1u_2u_3$ is heavy since all of its arcs are diagonals, a contradiction. If $S$ has a triangle, we can repeat the same argument. 
\end{proof}

\begin{algorithm}[H]\caption{\textsc{Layers}$(T,w, U_i, W$)}
\label{layersalgo}
\begin{algorithmic}[1]
\REQUIRE $T$ is a light tournament, $w:V(T) \to \Q_{\geq 0}$, $U_i$ is the current root layer, and $W$ is the set of unseen vertices ($U_0:=\varnothing$ and $W:=V(T)$ on the first call). We assume all objects that depend on $i$ (including $i$ itself) to be available throughout subsequent recursive calls.
\ENSURE A feedback vertex set $F'$ of $T$ of weight at most $\frac{3}{4}  w(T)$
\
\IF[Finished]{$W = \varnothing$}
    \STATE{$L_0 \gets \cup_{j\ \mathrm{even}} U_{j} \cup S_{j}$, $L_1 \gets \cup_{j\ \mathrm{odd}} U_{j} \cup S_{j}$}
    \STATE{$F' \gets (\cup_{j=1}^{i}F_{2j}) \cup L_1$ if $w(L_0) \geq w(L_1)$ otherwise $(\cup_{j=0}^{i-1}F_{2j+1}) \cup L_0$}
    \STATE{return $F'$}
\ENDIF
\IF{$N(U_i) \neq \varnothing$}
  \STATE{$\{z_{i},z_{i}'\} \gets $2-in-dominates$(N(U_i))$ with $w(N(z_{i}) \cap W) \geq w(N(z_{i}') \cap W)$}\label{line:split}
    \STATE{$U_{i+1} \gets N(z_{i}) \cap W$, $S_{i+1} \gets N(z_{i}') \cap W - U_{i+1}, W \gets W - U_{i+1} - S_{i+1}$}\label{line:s_i}
    \STATE{$F_{i+1} = S_{i+1}$}\label{line:f_i}
    \STATE{$i \gets i+1$}
    \STATE{return \sc{Layers}$(T, w, U_{i+1},W)$}
\ELSE[Fresh Start]
    \STATE{$z_{i+1} \gets$ choose $z \in W$ with $|N(z) \cap W|$ minimum}
    \STATE{$U_{i+1} \gets \{z_{i+1}\}$ , $U_{i+2} \gets N(z_{i+1}) \cap W, S_{i+1} \gets \varnothing$}\label{line:ui}
    \STATE{$F_{i+1} \gets \varnothing$}
    \STATE{$F_{i+2} \gets $ \sc{CDZ$(U_{i+2},w)$}}\label{line:freshstart}
    \STATE{$W \gets W - (U_{i+1} \cup U_{i+2})$}
    \STATE{$i \gets i+2$}
    \STATE{return \sc{Layers}$(T, w, U_{i+2},W)$}
\ENDIF  

\end{algorithmic}
\end{algorithm}

The next lemma ensures that, in the step following a fresh start, vertices $z_{i}, z'_{i}$ as on line \ref{line:split} of Algorithm \ref{layersalgo} exist.
\begin{lemma}\label{twodominatesthree}
For an arbitrary vertex $z$ in a light tournament $T$, $V_3(z)$ is 2-in-dominated by $V_2(z)$.
\end{lemma}
\begin{proof}
Let $H = \{h_1,h_2,...,h_k\} \subseteq V_2(z)$ be an inclusion-wise minimal set that in-dominates $V_3(z)$. Suppose $k \geq 3$. By minimality, for each $h_i \in H$ there must be some $v_i \in V_3(z)$ such that $(v_i,h_i) \in A(T)$ and $(h_i,v_j) \in A(T)$ for all $j \neq i$. Since $(z,v_i) \in A(T)$ for all $i$, it follows that $T[\{z,h_1,h_2,h_3,v_1,v_2,v_3\}]$ is isomorphic to a tournament in $\mathcal S_7$ (see Figure~\ref{fig:S7}). Therefore, by Lemma~\ref{lem:S_7}, $T[\{z,h_1,h_2,h_3,v_1,v_2,v_3\}]$ is heavy, which contradicts that $T$ is light.      
\end{proof}

The next lemma ensures that the layer produced after a fresh start is $\mathcal{T}_5$-free, allowing us to use the exact algorithm from~\cite{CDZ01}. Its proof follows the proof of Lemma 9 of \cite{MWV16}, except that we assume that $T$ is light.

\begin{lemma} \label{lem:H}
Let $z$ be a minimum in-degree vertex in a light tournament $T$. Then $V_2(z)$ is $\mathcal{T}_5$-free.
\end{lemma}

\begin{proof}
 We assume $V_2(z) \neq \varnothing$, otherwise there is nothing to show, and we suppose by contradiction that $X \subseteq V_2(z)$ is a light $\mathcal{T}_5$  ($X$ cannot be heavy as $T$ is light, so $X$ is oriented as in Figure~\ref{fig:lightT5}). For every $u \in V_2(z)$ there must be a $v \in V_3(z)$ with $(v,u) \in A(T)$. 
 If not then $N(u) \subsetneq V_2(z) = N(z)$, contradicting the minimality of $|N(z)|$. Thus $V_3(z) \neq \varnothing$. Let $H \subseteq V_3(z)$ be an inclusion-wise minimal subset of $V_3(z)$ such that for every $u \in V_{2}(z)$ there exists $v \in H$ with $(v,u) \in A(T)$. We distinguish cases according to the size of $H$.\medskip
 
\noindent \emph{Case 1: $H = \{h\}$}. Then $hu_iz$ are triangles for all $u_i \in X$, therefore all arcs in $X$ are diagonals. Since $X$ must contain at least some triangle, this triangle must be heavy since all of its arcs are diagonals, contradicting the fact that $T$ is light.\medskip

\noindent \emph{Case 2: $H = \{f,h\}$}. Let $X = \{a,b,c, d,e\}$.
We can assume without loss of generality that $f$ points to exactly three vertices of $X$, for the following reason. If there are less than three, we can swap $h$ with $f$. If there are more than three, then $f$ must point to a triangle of $X$ (since $T[X]$ is a $\mathcal{T}_5$-subtournament), which would be heavy, arguing as in Case~1. 

Notice that $ed$ and $ec$ are diagonals within $X$ (due to triangles $ade$ and $adc$, $bdc$ and $bec$, respectively), hence none of $ad, ae, bc, be$ can be diagonals, otherwise one of $ade$ or $cbe$ will be a heavy triangle. This implies that $f$ cannot point to both vertices of any of the latter pairs. From this, one easily derives that $f$ cannot point to $a$ nor $b$. Hence,  $(a,f)$, $(b,f)$, $(f,d)$, $(f,e)$, $(f,c) \in A(T)$, 
 which implies $(h,a), (h,b) \in A(T)$. This forces $(e,h),(c,h),(d,h) \in A(T)$; otherwise, again, one of $ad, ae, bc, be$ is a diagonal. 
See Figure~\ref{fig:lemH} for the orientations we have determined thus far.  
Notice that $adc$ and $fca$ are triangles, so $df$ is a diagonal. Moreover, since $had$ and $zha$ are triangles, $dz$ is a diagonal. Therefore $zfd$ is a heavy triangle, a contradiction.\medskip

\noindent \emph{Case 3: $|H|\geq 3$}. In this case, one can easily find a tournament in $\mathcal{S}_7$ made of $z$, three vertices of $V_2(z)$ and three vertices of $V_3(z)$, in contradiction with Lemma \ref{lem:S_7} (see the proof of Lemma \ref{twodominatesthree}). \qedhere

\begin{figure}[h]
\centering
\begin{tikzpicture}[scale=.75,inner sep=1.5pt]
\tikzstyle{vtx}=[circle,draw,thick];
\node[vtx] (z) at (0,2) {$z$};
\node[vtx] (a) at (-2,0) {$a$};
\node[vtx] (b) at (2,0) {$b$};
\node[vtx] (c) at (-2,-2) {$c$};
\node[vtx] (d) at (0,-2) {$d$};
\node[vtx] (e) at (2,-2) {$e$};
\node[vtx] (f) at (-3,-5) {$f$};
\node[vtx] (h) at (3,-5) {$h$};
\draw[thick,->,>=latex] (a) -- (z);
\draw[thick,->,>=latex] (b) -- (z);
\draw[thick,->,>=latex] (c) -- (z);
\draw[thick,->,>=latex] (d) -- (z);
\draw[thick,->,>=latex] (f) -- (c);
\draw[thick,->,>=latex] (f) -- (d);
\draw[thick,->,>=latex] (f) -- (e);
\draw[thick,->,>=latex] (a) -- (f);
\draw[thick,->,>=latex] (b) to[bend right=120] (f);
\draw[thick,->,>=latex] (h) to[bend right = 120] (a);
\draw[thick,->,>=latex] (h) -- (b);
\draw[thick,->,>=latex] (d) -- (h);
\draw[thick,->,>=latex] (e) -- (h);
\draw[thick,->,>=latex] (c) -- (h);
\draw[thick,->,>=latex] (e) -- (a);
\draw[thick,->,>=latex] (e) to[bend left = 50] (c);
\draw[thick,->,>=latex] (d) -- (e);
\draw[thick,->,>=latex] (b) -- (e);
\draw[thick,->,>=latex] (c) -- (a);
\draw[thick,->,>=latex] (c) -- (b);
\draw[thick,->,>=latex] (d) -- (c);
\draw[thick,->,>=latex] (a) -- (d);
\draw[thick,->,>=latex] (b) -- (d);
%\draw[thick] (2,2) -- (2.75,3);
\end{tikzpicture}
\caption{The orientations determined by the proof of Lemma \ref{lem:H}.}
\label{fig:lemH}
\end{figure}
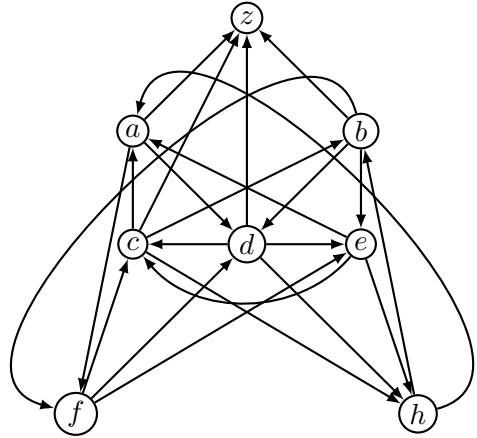
\end{proof}

The next lemma ensures that, at each recursive call of Algorithm \ref{layersalgo}, we can find local solutions by either applying \textsc{CDZ$(T,w)$} or Lemma \ref{lem:2-in-dom_implies_Triangle-free}. Together with the previous lemmas, it is enough to conclude that Algorithm \ref{layersalgo} outputs a feedback vertex set of our (light) tournament $T$. This will be formalized in Lemma \ref{lem:correct}.
\begin{lemma}
\label{lem:Ui-no-T5}
Let $U_0, \dots, U_\ell$ and $S_1, \dots, S_\ell$ be the sets produced by Algorithm \ref{layersalgo}, run on input $(T,w,U_0:=\varnothing, W:=V(T))$. For all $i \in [\ell-1]$, if $U_{i+1}$ is defined as on line \ref{line:ui} of Algorithm~\ref{layersalgo}, then $T[U_{i+1}]$ is $\mathcal{T}_5$-free; and if  $S_{i+1}$ is defined as on line \ref{line:s_i}, then $S_{i+1}$ is a feedback vertex set of $T[U_{i+1} \cup S_{i+1}]$. 
\end{lemma}
\begin{proof}
 If $U_{i+1}$ is defined as on line \ref{line:ui} of Algorithm~\ref{layersalgo}, then $U_{i+1}$ is equal to $V_{2}(z)$ for some $z\in V(T)$. Therefore, by Lemma \ref{lem:H}, $T[U_{i+1}] = T[V_2(z)]$ is $\mathcal{T}_5$-free. If  $S_{i+1}$ is defined as on line \ref{line:s_i}, then there is some vertex $z_{i-1}\in V(T)$ such that  $U_{i+1} \cup S_{i+1} \subseteq V_3(z_{i-1})$.  By Lemma \ref{twodominatesthree}, $N(U_{i})$ 2-in-dominates $U_{i+1} \cup S_{i+1}$. Therefore, by Lemma \ref{lem:2-in-dom_implies_Triangle-free}, $S_{i+1}$ and $U_{i+1}$ are both triangle-free. Thus, $S_{i+1}$ is a feedback vertex set of $T[U_{i+1} \cup S_{i+1}]$.
\end{proof}

After having shown the correctness of Algorithm \ref{layersalgo}, we focus on bounding the approximation ratio of its output. This mostly amounts to bounding the weight of the local solutions obtained during the algorithm.

\begin{lemma}\label{lem:halfweight}
Let $F_1, \dots, F_\ell$ and $U_0, \dots, U_\ell$ be the sets produced by Algorithm \ref{layersalgo}, run on input $(T,w,U_0:=\varnothing, W:=V(T))$. Then for all $i \in [\ell]$, $w(F_{i}) \leq w(N(U_{i-1}))/2$.
\end{lemma}
\begin{proof}
If $F_{i}= \emptyset$, then the lemma clearly holds.  
If $F_{i}$ is defined as $S_{i}$ on line \ref{line:f_i}, then, by construction $w(S_{i}) \leq w(N(U_{i-1}))/2$.
Thus, we may suppose that $F_{i}$ is defined as $ \textsc{CDZ}(U_{i},w)$ on line \ref{line:freshstart}, with $U_{i} = N(z_{i-1}) \cap W$, and $U_{i-1} = \{z_{i-1}\}$. By Lemma \ref{lem:Ui-no-T5}, $T[U_{i}]$ is $\mathcal{T}_5$-free, and by~\cite{CDZ01}, $F_{i}$ is a minimum weight feedback vertex set of $T[U_{i}]$. Since the all $\frac{1}{3}$-vector is feasible for the basic relaxation of $T[U_{i}]$, and this relaxation is integral by~\cite{CDZ01}, 
\[
w(F_{i}) \leq \frac{1}{3} w(U_{i})=\frac{1}{3} w(N(U_{i-1})) \leq  \frac{1}{2} w(N(U_{i-1})). \qedhere
\]
\end{proof}

In the next two lemmas, we assume that Algorithm ~\ref{alg:main} is run on input $(T,w)$, and we establish properties of the sets defined within the algorithm during the rounding phase (lines 1-\ref{line:until}).
\begin{lemma}\label{lem:iterative_rounding}
 After the rounding phase of Algorithm~\ref{alg:main},
 $$w(F) \leq \frac{7}{3} \big(\SA_1(T,w) - \SA_0(T-F-Z,w)\big).$$
\end{lemma}
\begin{proof}

We proceed by induction on the number of vertices added to $F$ on line~\ref{line:round_up_half}. In the base case, no vertices get added to $F$ on line~\ref{line:round_up_half}. Letting $x$ denote the optimal solution to $\SA_1(T,w)$, we get
\begin{align*}
w(F) &\le \frac{7}{3} \sum_{v \in F} w(v) x_v\\ 
&= \frac{7}{3} \left( \sum_{v \in V(T)} w(v) x_v  - \sum_{v \in V(T-F)} w(v) x_v \right)\\ 
&\le \frac{7}{3} \left( \sum_{v \in V(T)} w(v) x_v  - \sum_{v \in V(T-F-Z)} w(v) x_v \right)\\
&\le \frac{7}{3} \big(\SA_1(T,w) - \SA_1(T-F-Z,w)\big)\\
&\le \frac{7}{3} \big(\SA_1(T,w) - \SA_0(T-F-Z,w)\big)\,.
\end{align*}

Now let $F_{\mathrm{before}}$, $Z_{\mathrm{before}}$, $F_{\mathrm{after}}$, $Z_{\mathrm{after}}$ denote the sets $F$ and $Z$ before and after a single iteration of the loop in lines \ref{line:repeat}--\ref{line:until}. Using an argument similar as the one used above (arguing this time with an optimal solution $x$ to $
\SA_0(T-F_{\mathrm{before}}-Z_{\mathrm{before}},w)$), we get
\begin{align*}
w(F_{\mathrm{after}}) - w(F_{\mathrm{before}}) &\leq \frac{1}{2} \big(\SA_0(T-F_{\mathrm{after}}-Z_{\mathrm{after}},w) - \SA_0(T-F_{\mathrm{before}}-Z_{\mathrm{before}},w)\big)\\
&\leq \frac{7}{3} \big(\SA_0(T-F_{\mathrm{after}}-Z_{\mathrm{after}},w) - \SA_0(T-F_{\mathrm{before}}-Z_{\mathrm{before}},w)\big).
\end{align*}

Hence, assuming that
$$
w(F_{\mathrm{before}}) \leq \frac{7}{3} \big(\SA_1(T,w) - \SA_0(T-F_{\mathrm{before}}-Z_{\mathrm{before}},w)\big)\,,
$$
we get
$$
w(F_{\mathrm{after}}) \leq \frac{7}{3} \big(\SA_1(T,w) - \SA_0(T-F_{\mathrm{after}}-Z_{\mathrm{after}},w)\big)\,.
$$
The result follows.
\end{proof}

\begin{lemma}\label{lem:thirdweight}
After the rounding phase of Algorithm~\ref{alg:main},
$$\SA_0(T-F-Z,w) = w(T-F-Z)/3.$$
\end{lemma}
\begin{proof}
The proof is the same as \cite[Lemma 6]{MWV16}, but for completeness, we include it here. Let $T'=T-F-Z$.  Suppose $x_v = 0$ for some $v \in V(T')$.  Since every vertex of $T'$ is contained in a triangle, $v$ is in some triangle $vab$ of $T$.  Thus, $x_a+x_b \geq 1$, and so $\max(x_a,x_b) \geq 1/2$, which  contradicts that neither $a$ nor $b$ are in $F$.   Thus $x_v > 0$ for all $v \in V(T')$.  Let $y$ be an optimal solution to the dual of $\SA_{0}(T',w)$.  By primal-dual slackness $\sum_{\triangle: u \in \triangle} y_{\triangle}=w_u$ for all $u \in V(T')$.  Therefore,
\[
w(V(T'))=\sum_{u \in V(T')} \sum_{\triangle: u \in \triangle} y_{\triangle}= \sum_{\triangle \in \triangle(T')} y_{\triangle} \sum_{u \in \triangle} 1 = 3 \sum_{\triangle \in \triangle(T')} y_{\triangle}= 3 \SA_0(T', w). \qedhere
\]
\end{proof}

\begin{lemma}\label{lem:correct}
 Let $F'$ be the set output by Algorithm \ref{layersalgo} on input $(T':=T-F-Z, w, U_0:=\varnothing, W=V(T'))$. Then $F \cup F'$ is a feedback vertex set of $T$ and $w(F') \leq \frac{9}{4} \SA_0(T',w)$.
\end{lemma}
\begin{proof}

Algorithm~\ref{layersalgo} partitions $V(T')$ into layers $S_i\cup U_i$. By symmetry, we may assume that the total weight of the even layers is at least the total weight of the odd layers.  That is, $w(L_0) \geq w(L_1)$, using the notation of the algorithm. Then the output $F'$ consists of all odd layers and of the sets $F_i$, for $i$ even. By construction, $F_i$ is an FVS of $T'[S_i \cup U_i]$, for each $i$.  % Lemma \ref{lem:Ui-no-T5} if $F_i = $ \textsc{CDZ$(U_i)$}, and by Lemma \ref{lem:2-in-dom_implies_Triangle-free} if $F_i = S_i$.
Since all triangles in $T'$ are contained in three consecutive layers, $F'$ is an FVS of $T'$, and hence $F \cup F'$ is an FVS of $T$.  Moreover, since $w(F_i) \leq w(S_i \cup U_i)/2$ for each $i$, we have $w(L_0) - w(\cup_{\text{$j$ even}}F_{j}) \geq w(V(T'))/4$.  Therefore, 
\[
w(F') = w(V(T')) - (w(L_0) - w(\cup_{\text{$j$ even}}F_{j})) \leq \frac{3}{4} w(V(T')) = \frac{9}{4} \SA_0(T',w),
\]
where the last equality follows from Lemma \ref{lem:thirdweight}.
\end{proof}

\section{The Algorithm} \label{sec:algorithm}
Given the results we have already established, it is now easy prove the correctness of Algorithm~\ref{alg:main}.  

\begin{reptheorem}{thm:main}
Algorithm \ref{alg:main} is a $7/3$-approximation algorithm for \fvst. More precisely, the algorithm outputs in polynomial time a feedback vertex set $X := F \cup F'$ such that $w(X) \leq \frac{7}{3} \SA_1(T,w) \leq \frac{7}{3} \OPT(T,w)$.
\end{reptheorem}
\begin{proof}
By Lemma~\ref{lem:correct}, $F \cup F'$ is a feedback vertex set of $T$. It remains to show the approximation guarantee. Recall that $F = \{v: x_v \geq 3/7\}$ where $x$ is an optimal solution for $\SA_1(T,w)$. By Lemma~\ref{lem:iterative_rounding},  $w(F) \leq \frac{7}{3}(\SA_1(T,w) - \SA_0(T-F-Z,w))$.  Since $x$ restricted to $T-F-Z$ is feasible for $\SA_1(T-F-Z)$, Lemma~\ref{lem:correct} implies that $w(F') \leq \frac{9}{4} \SA_0(T-F-Z,w) \leq \frac{7}{3} \SA_0(T-F-Z,w)$.  Adding these two inequalities yields
\[
    w(F') + w(F) \leq \frac{7}{3} \SA_1(T,w) \leq \frac{7}{3} \OPT(T,w). \qedhere
\]
\end{proof}

Finally, as mentioned in the introduction, we have the following corollary on the integrality gap of $\SA_1$ for \fvst{}, which we now formally define.  If $T$ is a tournament and $w : V(T) \to \R_{\geq 0}$, we let
$$
\SA_r(T,w) := \min \left\{\sum_{v \in V(T)} w(v) x_v \mid x \in \SA_r(T)\right\}.
$$
The (worst case) \emph{integrality gap} of $\SA_r$ for \fvst{} is 
$$
\sup_{(T,w)} \frac{\OPT(T,w)}{\SA_r(T,c)}
$$
where the supremum is taken over all tournaments $T$ and all weight functions $w : V(T) \to \Q_{\geq 0}$.

\begin{corollary} \label{cor:integrality_gap}
The integrality gap of $\SA_1$ for \fvst{} is exactly $7/3$.
\end{corollary}

\begin{proof}
The fact that the integrality gap of $\SA_1$ for \fvst{} is at most $7/3$ follows from Theorem~\ref{thm:main}. For the other inequality, note that for every tournament $T$, the all $\frac{3}{7}$-vector is feasible for $\SA_1(T)$ (by setting $x_{uv}=\frac{1}{7}$ for all $uv$).  On the other hand, it is easy to show via the probabilistic method that for a random $n$-node tournament $T$, $\OPT(T,\mathbf{1}_T) = n - O(\log n)$ with high probability.
\end{proof}

\section{Conclusion}\label{sec:conclusion}
In this paper we give a simple $7/3$-approximation algorithm for \fvst{}, based on performing just one round of the Sherali-Adams hierarchy on the basic relaxation.  It is a bit of a miracle that $\SA_{1}(T)$ already ``knows'' a remarkable amount of structure about feedback vertex sets in tournaments. It is unclear how much more knowledge $\SA_{r}(T)$ acquires as $r$ increases, but our approach naturally begs the question of whether performing a constant number of rounds of Sherali-Adams leads to a $2$-approximation for \fvst. This would solve the main open question from~\cite{LMMPPS20}.

We suspect that performing more rounds does improve the approximation ratio, but the analysis becomes more complicated.  Indeed,  it could be that $\SA_{2}(T)$ already gives a $9/4$-approximation algorithm for $\fvst{}$, since our layering procedure has a $9/4$-approximation factor,
Note that $\SA_2(T)$ does contain new inequalities such as $x_{a} + x_{b} + x_{c} \geq 1 + x_{ab} + x_{ac} + x_{cb} - x_{abc}$, for all $abc \in \triangle(T)$, which may be exploited.  

As further evidence, for the related problem of \emph{cluster vertex deletion}~\cite{ADFH20}, we showed that  one round of Sherali-Adams has an integrality gap of $5/2$, and for every $\epsilon >0$ there exists $r \in \mathbb N$ such that $r$ rounds of Sherali-Adams has integrality gap at most $2+\epsilon$.  Indeed, this work can be seen as unifying the approaches of \cite{MWV16} and some of the polyhedral results of~\cite{ADFH20}. 

\bibliographystyle{abbrv}
\bibliography{references}

\end{document}